\documentclass[a4paper,11pt]{amsart}

\usepackage{fullpage}
\usepackage{amsmath,amsthm,amssymb,mathtools,mathrsfs}
\usepackage{hyperref}

\usepackage{chessfss}
\usepackage{tikz-cd}
\usepackage{tikz}
\usetikzlibrary{intersections,calc,arrows.meta}

\usepackage{graphicx}
\usepackage{here}
\usepackage{time}

\usepackage{xcolor}
\usepackage[capitalize,nameinlink,noabbrev,nosort]{cleveref}
\hypersetup{
	colorlinks=true,
	linkcolor=blue,
	citecolor=red,
	filecolor=brown,
	urlcolor=brown,
}

\makeatletter
\@namedef{subjclassname@2020}{%
  \textup{2020} Mathematics Subject Classification}
\makeatother

\newtheorem{theoremcounter}{Theorem Counter}[section]

\theoremstyle{definition}
\newtheorem{dfn}[theoremcounter]{Definition}
\newtheorem{remark}[theoremcounter]{Remark}

\theoremstyle{plain}
\newtheorem{lem}[theoremcounter]{Lemma}
\newtheorem{proposition}[theoremcounter]{Proposition}
\newtheorem{corollary}[theoremcounter]{Corollary}

\newtheorem{thm}[theoremcounter]{Theorem}

\numberwithin{equation}{section}

\newcommand{\Z}{\mathbb{Z}}
\newcommand{\Q}{\mathbb{Q}}
\newcommand{\R}{\mathbb{R}}

\newcommand{\Ker}{\operatorname{Ker}}

\DeclareMathOperator{\ImNew}{Im}
\renewcommand{\Im}{\ImNew}

\newcommand{\mc}[1]{\mathcal{#1}}

\DeclareMathOperator{\Res}{Res}
\DeclareMathOperator{\Spec}{Spec}

\DeclareMathOperator{\lcm}{lcm}

\begin{document}

\title[]{Non-vanishing of the $p$-adic constant for mock modular forms associated to a newform with real Fourier coefficients} 

\author{Ryota Tajima} 
\address{Faculty of Mathematics, Kyushu University, 744 Motooka, Nishi-ku, Fukuoka, 819-0395, Japan}
\email{ryota.tajima.123@gmail.com}


\maketitle

\begin{abstract}
Let $F^{+}$ be a mock modular form associated to a normalized newform $g$. K. Bringmann et~al.\@ \cite{bringmann2012mock} obtained a $p$-adic modular form starting from $F^{+}$ by adding a suitable linear combination of Eichler integrals of $g(q)$ and $g(q^{p})$. We denote the coefficients of the Eichler integrals of $g(q)$ and $g(q^{p})$ by $\gamma_{g}$ and $\delta_{g}$. These constants are important in the $p$-adic theory of mock modular forms, but relatively little is known about them at present. For instance, K. Bringmann et~al.\@ \cite{bringmann2012mock} raised the question of whether $\delta_{g}$ vanishes when $g$ has CM by an imaginary quadratic field in which $p$ is inert. In previous work, the non-vanishing of $\delta_{g}$ has been proved mainly when $g$ is associated to an elliptic curve. In higher weight, only one example was known for which $\delta_{g}\neq 0$. In this paper, we show that $\delta_{g}\neq 0$ under mild assumptions when all the Fourier coefficients of $g \in S_{k}(\Gamma_{0}(N), \chi)$ are real, without assuming that $g$ has CM. In particular, this provides a class of higher-weight examples for which $\delta_{g}\neq 0$.

\end{abstract}

\section{Introduction}
Modular forms have a rich theory and are related to several objects, such as the sum of divisors, elliptic curves, and Galois representations. In 1920, Ramanujan discovered a class of functions that is similar to modular forms, and he called them \emph{mock theta functions}. After this, S. Zwegers \cite{zwegers2008mock} incorporated mock theta functions into the theory of harmonic Maass forms. In simple terms, a harmonic Maass form $F$ is a real-analytic function that satisfies the modular transformation law (see Definition \ref{def of harmonic maass}). A harmonic Maass form $F$ is written as the sum of the \emph{holomorphic part} $F^{+}$ and the \emph{non-holomorphic part} $F^{-}$. It is known that Ramanujan's mock theta function can be regarded as the holomorphic part of a harmonic Maass form. Therefore, we call the holomorphic part of a harmonic Maass form \emph{mock modular form} following D. Zagier \cite{zagier2009ramanujan}.

We recall how to associate mock modular forms to a normalized newform $g \in S_{k}(\Gamma_{0}(N), \chi, K)$, where $K$ is a number field following \cite{bruinier2008differential}. In general, coefficients of mock modular forms seem to be transcendental. Therefore, it appears to be challenging to relate mock modular forms to other number-theoretic topics. However, J. H. Bruinier, K. Ono, and R. C. Rhoades \cite{bruinier2008differential} attached a normalized newform $g$ to mock modular forms which have algebraic properties. Let $\xi_{2-k}:=2iy^{2-k}\overline{( \partial/\partial \overline{\tau})}$, where $\tau \in \mathbb{H}$ and $\tau=x+iy$. Then, $\xi_{2-k}$ gives a surjective map from harmonic Maass forms to cusp forms. Hence, we consider lifts of $g$ by $\xi_{2-k}$. J. H. Bruinier et~al.\@  \cite{bruinier2008differential} showed that there exist lifts of $g$ that satisfy certain algebraic properties, and they call them \emph{good for $g$}. (See Definition \ref{good on Gamma1}.)  In this paper, we also call its holomorphic part a \emph{good lift of $g$}. We note that a good lift of $g$ is a mock modular form. A remarkable property of a good lift $F^{+}$ is that we can modify it to a $K$-coefficient $q$-series. We explain it more precisely. For a $q$-series $h(q)=\sum_{n>0}a_{h}(n)q^{n}$, we define the \emph{Eichler integral} $E_{h}(q)$ by a $q$-series $\sum_{n>0} n^{1-k}a_{h}(n)q^{n}$. Then, the following $q$-series is in $K(\!(q)\!)$:
\begin{align*}
F^{+}(q)-a_{F^{+}}(1) E_{g}(q)= \sum _{n\gg -\infty}\left( a_{F^{+}}(n) -a_{F^{+}}(1)n^{1-k}a_{g}(n) \right) q^{n}.
\end{align*}

Next, we introduce the $p$-adic constants $\gamma_{g}$ and $\delta_{g}$ associated to a good lift. Let $g \in S_{k}(\Gamma_{0}(N), \chi, K)$ be a normalized newform with real Fourier coefficients, i.e. $a_{g}(n) \in \mathbb{R}$ for all $n>0$. We take a good lift of $g$ and denote it by $F^{+}$. For a complex number $a_{F^{+}}(1)$, we define $a_{F^{+}}(1)+K_{p}$ to be a set of all formal pairs $(a_{F^{+}}(1), \alpha)$ where $\alpha \in K_{p}$. We denote an element $(a_{F^{+}}(1), \alpha)$ by $a_{F^{+}}(1)+\alpha$. We consider a $2$-dimensional $p$-adic cohomological space $M_{\mathrm{dR}, p}(g)$ attached to $g$. This space is equipped with an operator $V$ whose action on $q$-expansions is given by $q \mapsto q^{p}$.  K. Bringmann, P. Guerzhoy, and B. Kane \cite{bringmann2012mock} constructed the following $q$-series for $\gamma=a_{F^{+}}(1)+\alpha \in a_{F^{+}}(1)+K_{p}$ and $\delta \in K_{p}$:
 \begin{align*}
     \mathcal{F}_{\gamma, \delta} &\coloneqq F^{+}-\gamma E_{g(q)}-\delta E_{g(q^{p})}
     \\
     &\coloneq \left(F^{+}-a_{F^{+}}(1)E_{g(q)} \right)-\alpha E_{g(q)}-\delta E_{g(q^{p})}.
 \end{align*}
We note that all coefficients of $\mathcal{F}_{\gamma, \delta}$ are in $K_{p}$. $\mathcal{F}_{\gamma, \delta}$ is just a formal power series in general. However,  K. Bringmann et~al.\@ \cite{bringmann2012mock} and L. Candelori and F. Castella \cite{candelori2017geometric} showed that $\mc{F}_{\gamma_{g}, \delta_{g}}$ is a $q$-expansion of a $p$-adic modular form when $\{g, V(g)\}$ is a basis of $M_{\mathrm{dR}, p}(g)$. The author thinks that the $p$-adic modular form $\mc{F}_{\gamma_{g}, \delta_{g}}$ is the key to developing a $p$-adic theory of mock modular forms. Hence, we investigate the $p$-adic constants $\gamma_{g}$ and $\delta_{g}$ as a first step. Thanks to the previous work, it is known that $\gamma_{g}=0$ holds when $g$ has complex multiplication by an imaginary quadratic field in which $p$ is inert. In other words, the following $q$-series is a $q$-expansion of a $p$-adic modular form:
\begin{align*}
\mathcal{F}_{0, \delta_{g}}= F^{+}-\delta_{g} E_{g(q^{p})}.
     \end{align*}
In this case, the $p$-adic constant $\delta_{g}$ may be regarded as a correction term for a mock modular form to become a $p$-adic modular form. Hence, it is natural to ask whether $\delta_{g}$ vanishes. In fact, K. Bringmann et~al.\@ \cite{bringmann2012mock} said that ``It is an interesting question whether $\alpha=0$ ever occurs in Theorem 1.3." (See \cite[Remarks (2) following Theorem 1.3]{bringmann2012mock}.) Here $\alpha$ denotes $\delta_{g}$ in our notation. In this paper, we answer this question when $g$ has real Fourier coefficients. Moreover, we show that $\delta_{g}$ does not vanish without assuming that $g$ has complex multiplication. The following is the main theorem:
\begin{thm}
 Let $g \in S_{k}(\Gamma_{0}(N), \chi, K)$ be a normalized newform with real Fourier coefficients. Suppose that $p\nmid 6N$, and $N \geq 5$. We also assume that $\{g, V(g)\}$ is a basis of $M_{\mathrm{dR}, p}(g)$. Then, $\delta_{g} \not=0$ holds.
\end{thm}
As a corollary, we answer the question in \cite[Remarks (2) following Theorem 1.3]{bringmann2012mock} when $g$ has real Fourier coefficients.
\begin{corollary}
 Let $g \in S_{k}(\Gamma_{0}(N), \chi, K)$ be a normalized newform with real Fourier coefficients. Suppose that $p\nmid 6N$, and $N \geq 5$. We also assume that $g$ has complex multiplication by an imaginary quadratic field in which $p$ is inert. Then, $\delta_{g} \not=0$ holds.
\end{corollary}

Next, we recall known results on $\delta_{g}$ and compare them with the main theorem. The first is due to K. Bringmann et~al.\@ \cite{bringmann2012mock}. Let $g$ be the unique normalized newform in $S_{4}(\Gamma_{0}(9))$. Then, $g$ has complex multiplication by $\Q(\sqrt{-3})$. K. Bringmann et~al.\@ \cite{bringmann2012mock} showed that $\delta_{g}$ is not zero for all inert primes $p$ such that $p < 32500$. After this, the author \cite{tajima2023p} showed that $\delta_{g}$ is a $p$-adic unit when $g \in S_{k}(\Gamma_{0}(N))$ has complex multiplication by an imaginary quadratic field $L$ in which $p$ is inert, and $\dim S_{k}(\Gamma_{0}(N))=1$ holds.  After this, P. Guerzhoy \cite{guerzhoy2025non} essentially showed that $\delta_{g}$ is not zero for almost all prime $p$ such that $p \nmid N$ and $a_{g}(p)=0$ when $g \in S_{2}(\Gamma_{0}(N), \mathbb{Q})$ (cf. \cite[Theorem 1]{guerzhoy2025non}). In the case that $g \in S_{2}(\Gamma_{0}(N), \mathbb{Q})$ has complex multiplication by $L$ and $p$ is inert in $L$, the author \cite{tajima2025p} determined the $p$-adic valuation of $\delta_{g}$ under mild assumption for $p$. Both papers \cite{guerzhoy2025non} and \cite{tajima2025p} use the formal group theory of elliptic curves. Therefore, it does not seem easy to prove the higher weight case using the same method. In this paper, we show that $\delta_{g}$ is not zero even in the higher weight case by using the cohomological interpretation of the differential of a good lift of $g$ developed in \cite{candelori2014harmonic} and \cite{candelori2017geometric}.

We briefly explain the proof of the main theorems. For simplicity, we only consider the case that $g$ has complex multiplication. Let $\mathcal{L}^{1}$ be the relative de Rham cohomology sheaf of the universal generalized elliptic curve $\mc{E}_{\mathrm{univ}}^{\mathrm{gen}} \rightarrow X=X_{1}(N)/K$. It is known that there is a connection $\nabla : \mathcal{L}^{1} \rightarrow \mathcal{L}^{1} \otimes \Omega_{X}^{1}(\log C)$, where $C$ is the cuspidal subscheme (see \cite[Section 2]{candelori2017geometric}). We denote the parabolic cohomology of $(\mathcal{L}^{k-2}, \nabla)$ by $\mathbb{H}^{1}_{\mathrm{par}}(X, \mathcal{L}^{k-2})$. L. Candelori \cite{candelori2014harmonic} showed that $\mathbb{H}^{1}_{\mathrm{par}}(X, \mathcal{L}^{k-2})$ can be expressed in terms of weakly holomorphic modular forms. We denote the space of all weakly holomorphic modular forms with $K$-coefficients by $M_{k}^{!}(\Gamma_{0}(N), K)$. We say that $h \in M_{k}^{!}(\Gamma_{0}(N), K)$ is a \emph{weakly holomorphic cusp form} with $K$-coefficients if the constant term vanishes at any cusp. We denote the space of all weakly holomorphic cusp form with $K$-coefficients by $S^{!}_{k}(\Gamma_{1}(N), K)$.  L. Candelori \cite{candelori2014harmonic} showed that the following isomorphism holds:

\begin{align*}
    \mathbb{H}^{1}_{\mathrm{par}}(X, \mathcal{L}^{k-2}) \simeq \dfrac{S^{!}_{k}(\Gamma_{1}(N), K)}{D^{k-1}M^{!}_{2-k}(\Gamma_{1}(N), K)},
\end{align*}
where $D\coloneqq qd/dq$. For a good lift $F^{+}$, it is known that $D^{k-1}(F^{+})$ is a weakly holomorphic cusp form. In addition, all coefficients of  $D^{k-1}(F^{+})$ are in $K$ because $g$ has complex multiplication. Hence, we denote the class in $\mathbb{H}^{1}_{\mathrm{par}}(X, \mathcal{L}^{k-2})$ corresponding to $D^{k-1}(F^{+})$ by $[D^{k-1}(F^{+})]$. We define the $g$-isotypical component of $\mathbb{H}^{1}_{\mathrm{par}}(X, \mathcal{L}^{k-2})$ via Hecke action and denote it by $M_{\mathrm{dR}}(g)$.  Since the parabolic cohomology is stable under base change to a field with characteristic $0$, $M_{\mathrm{dR}, p}(g)\coloneqq M_{\mathrm{dR}}(g) \otimes K_{p}$ naturally lies in $\mathbb{H}^{1}_{\mathrm{par}}(X, \mathcal{L}^{k-2}) \otimes K_{p}=\mathbb{H}^{1}_{\mathrm{par}}(X_{1}(N)_{K_{p}}, \mathcal{L}^{k-2})$. L. Candelori and F. Castella \cite{candelori2017geometric} defined an endomorphism $V$ on $M_{\mathrm{dR}, p}(g)$ whose action on $q$-expansions is $q \mapsto q^{p}$. Since we assume that $\{g, V(g) \}$ is a basis of $M_{\mathrm{dR}, p}(g)$, there exist unique $p$-adic numbers $A, B \in K_{p}$ such that the following equality holds:

\begin{align*}
[D^{k-1}(F^{+})]=Ag+BV(g).
\end{align*}

L. Candelori et~al.\@ \cite{candelori2017geometric} showed that $\gamma_g=A$ and $\delta_{g}=-B$ hold. Hence, it is enough to show that $\{[D^{k-1}(F^{+})], g\}$ is a basis of  $M_{\mathrm{dR}, p}(g)$. Since both $D^{k-1}(F^{+})$ and $g$ are defined over $K$, it is sufficient to show that $\{[D^{k-1}(F^{+})], g\}$ is a basis of $M_{\mathrm{dR}}(g)$. We define a pairing $ \langle \cdot, \cdot\rangle$ on  $\mathbb{H}^{1}_{\mathrm{par}}(X, \mathcal{L}^{k-2})$ referring to the pairing defined by J. H. Bruinier and J. Funke \cite{bruinier2004two}. For $[\phi], [\psi] \in  \mathbb{H}^{1}_{\mathrm{par}}(X, \mathcal{L}^{k-2})$, we define
 \begin{align*}
 \langle [\phi], [\psi] \rangle \coloneqq \sum _{s \colon \text{cusp}}\sum _{n \in \mathbb{Z}\backslash\{0\}}\dfrac{a_{\phi }^{(s)}(n)a_{\psi }^{(s)}(-n)}{n^{k-1}}.
 \end{align*}
 By the Stokes' Theorem, we show that $\langle [D^{k-1}(F^{+})], g \rangle=(\xi_{2-k}(F), g)$ where $(\cdot, \cdot)$ is the Petersson inner product (see \cite{bruinier2004two}). By the definition of a good lift, we have that $\xi_{2-k}(F)=g/\left\| g\right\| ^{2}$. Therefore, $\langle [D^{k-1}(F^{+})], g \rangle=1$ holds. This implies that $[D^{k-1}(F^{+})]$ and $g$ are linearly independent because $\langle g, g \rangle=0$ holds. Hence, we conclude that $\delta_{g}\not=0$.

\begin{remark}
If not all the Fourier coefficients of $g$ are real, then $D^{k-1}(F^{+})$ is not in $M_{\text{dR}}(g)$. By contrast, $D^{k-1}((F^{c})^{+})$ lies in $M_{\text{dR}}(g)$ where $F^{c}$ is good for $g^{c}=\sum \overline{a_{g}(n)}q^{n} \in S_{k}(\Gamma_{0}(N), \overline{\chi})$. Therefore, we need to consider $F^{c}$ instead of $F$. However, $\langle [D^{k-1}((F^{c})^{+})], g \rangle=(\xi_{2-k}(F^{c}), g)=(g^{c}/\left\| g\right\| ^{2}, g)=0$ holds in general. This means that our method is not valid in the case that $g\not=g^{c}$ .
\end{remark}
Throughout this paper, we fix embeddings $\overline{\mathbb{Q}} \hookrightarrow \mathbb{C}$ and $\overline{\mathbb{Q}} \hookrightarrow \mathbb{C}_{p}$ and a valuation $v_{p}$ such that $v_{p}(p)=1$ and assume that $N \geq 5$ and $p \geq 5$.

This paper is based on the author’s doctoral thesis.

\section*{Acknowledgements}
The author would like to express his sincere gratitude to Prof.\ Shinichi Kobayashi for his helpful and valuable comments on the author's research. Discussions with him greatly improved the author's research. The author is grateful to Prof.\ Francesc Castella for helpful discussions on his paper \cite{candelori2017geometric} and for answering the author's questions about it. The author is also grateful to Prof.\ Ming-Lun Hsieh for providing the author with opportunities to discuss with Prof. Francesc Castella at National Taiwan University. The author would like to thank Prof.\ Bryden Cais. Thanks to discussions with him, the author gained a deeper understanding of the geometry of modular curves. The author would like to thank Prof.\ Kentaro Nakamura for giving the author valuable comments on the author's research. The author would also like to thank Prof.\ Takashi Hara for his strong interest in the author's research and for asking insightful mathematical questions. The author would like to thank Prof.\ Toshiki Matsusaka for providing valuable comments on analytic aspects of mock modular forms. The author was supported by JSPS KAKENHI Grant Number JP23KJ1720 and WISE program (MEXT) at Kyushu University.

\section{Analytic aspects of mock modular forms}
In this section, we review the analytic aspects of mock modular forms.
\subsection{Harmonic Maass forms and mock modular forms}
In this subsection, we recall known results on harmonic Maass forms and mock modular forms. 

We denote the upper-half plane by $\mathbb{H}$. For a $q$-series $f$, we denote the $n$-th coefficient of $f$ by $a_{f}(n)$. 
\begin{dfn}[{\cite[Definition 1]{candelori2014harmonic}}]\label{def of harmonic maass}
Let $k \in \mathbb{Z}$ be an even number. If a $C^{\infty}$ function $F$ satisfies the following conditions, we say that $F$ is a \emph{harmonic Maass form of weight $k$ on $\Gamma_{0}(N)$ with character $\chi$}:
\begin{itemize}
 \setlength{\leftskip}{0.5cm}

    \item[(1)] $F\left(\dfrac{a\tau+b}{c\tau+d}\right)=\chi(d)(c\tau+d)^{k}F(\tau)$ for all $\gamma =  \bigl(
\begin{smallmatrix}
   a & b \\
   c & d
\end{smallmatrix}
\bigr) \in \Gamma_{0}(N)$.
    \item[(2)] $\Delta_{k}F=0$, where
 $\Delta_{k}=-y^{2}\left( \dfrac{\partial ^{2}}{\partial x^{2}}+\dfrac{\partial ^{2}}{\partial y^{2}}\right) +iky\left( \dfrac{\partial }{\partial x}+i\dfrac{\partial }{\partial y}\right) $.
    \item[(3)] For any $\gamma =  \bigl(
\begin{smallmatrix}
   a & b \\
   c & d
\end{smallmatrix}
\bigr) \in SL_{2}(\mathbb{Z})$, there exist a natural number $h \in \mathbb{N}$, a polynomial $P_{\gamma} \in \mathbb{C}[T]$, and a positive number $\varepsilon  >0$ such that
\begin{align*}
\chi(d)^{-1}(c\tau+d)^{-k}F\left(\dfrac{a\tau+b}{c\tau+d}\right)-P_{\gamma}(\exp(-2\pi i \tau/h))=\text{O} (e^{-\varepsilon y})  \text{ as } y  \rightarrow \infty
\end{align*}

\end{itemize}
\end{dfn}
We denote a vector space of all harmonic Maass forms of weight $k$ on $\Gamma_{0}(N)$ with $\chi$ by $H_{k}(\Gamma_{0}(N), \chi)$ and we set $H_{k}(\Gamma_{1}(N)):=\oplus H_{k}(\Gamma_{0}(N), \chi)$. A harmonic Maass form has a Fourier expansion. It is known that a Fourier expansion of $F \in H_{k}(\Gamma_{1}(N))$ is written by
\begin{align}\label{q-expansion of harmonic maass}
F(\tau)=\sum _{n\gg -\infty}a_{F^{+}}( n) q^{n}+\sum _{n <0}a_{F^{-}}( n) \Gamma (1-k,4\pi  \left| n\right| y) q^{n},
\end{align} 
where $\Gamma(s, z):=\int ^{\infty }_{z}e^{-t}t^{s}\dfrac{dt}{t}$ is the incomplete Gamma-function (cf.\ \cite[Lemma 4.3]{bringmann2017harmonic}). We define the \emph{holomorphic part} $F^{+}(\tau)$ and the \emph{non-holomorphic part} $F^{-}(\tau)$ by
\begin{align*}
&F^{+}(\tau):=\sum _{n\gg -\infty}a_{F^{+}}( n) q^{n},
\\&
F^{-}(\tau):=\sum _{n <0}a_{F^{-}}( n) \Gamma ( k-1,4\pi  \left| n\right| y) q^{n}.
\end{align*}
Especially, the partial sum $\sum_{n \leq 0}a_{F^{+}}(n)q^{n}$ is called the \emph{principal part of $F$ at $\infty$}.

We denote the vector space of all weakly holomorphic modular forms of weight $k$ on $\Gamma_{1}(N)$ by $M^{!}_{k}(\Gamma_{1}(N))$. In other words, an element $h \in M^{!}_{k}(\Gamma_{1}(N))$ is a meromorphic modular form whose poles lie in only cusps. In particular, we say that $h$ is a \emph{weakly holomorphic cusp form} if $h \in M^{!}_{k}(\Gamma_{1}(N))$ and the constant term of its $q$-expansion is equal to $0$ at any cusp. We denote the space of weakly holomorphic cusp forms by $S^{!}_{k}(\Gamma_{1}(N))$. We can regard $h \in M^{!}_{k}(\Gamma_{1}(N))$ as an element in $H_{k}(\Gamma_{1}(N))$ whose non-holomorphic part $h^{-}$ is equal to $0$. 

Next, we introduce two important differential operators. There are two operators $D:=\frac{1}{2 \pi i}\frac{d}{d\tau}$ and $\xi_{2-k}:=2iy^{2-k}\overline{( \frac{\partial }{\partial \overline{\tau}})}$ on $H_{2-k}(\Gamma_{1}(N))$. These operators give the following maps:
\begin{align*}
&D^{k-1} : H_{2-k}(\Gamma_{1}(N)) \rightarrow S^{!}_{k}(\Gamma_{1}(N)),
\\&
\xi_{2-k} : H_{2-k}(\Gamma_{1}(N)) \rightarrow S_{k}(\Gamma_{1}(N)).
\end{align*}
Especially, the map $\xi_{2-k}$ is an antilinear morphism. In other words, $\xi_{2-k}(cF)=\overline{c}\xi_{2-k}(F)$ holds for all $c \in \mathbb{C}$. Therefore, $\xi_{2-k}$ maps from $H_{2-k}(\Gamma_{0}(N), \overline{\chi})$ to $S_{k}(\Gamma_{0}(N), \chi)$. We note that $D^{k-1}(F)=D^{k-1}(F^{+})$ and $\xi_{2-k}(F)=\xi_{2-k}(F^{-})$ hold. In other words, $D^{k-1}$ (resp. $\xi_{2-k}$) extracts information only from the holomorphic part $F^{+}$ (resp. the non-holomorphic part $F^{-}$). We say that $g \in S_{k}(\Gamma_{1}(N))$ is a normalized newform if $g\in S_{k}(\Gamma_{1}(N))^{\text{new}}$ and $g$ is a Hecke eigenform, and $a_{g}(1)=1$. We associate harmonic Maass forms with a normalized newform.

 \begin{dfn}[{\cite[Definition 1.1]{candelori2017geometric}}]\label{good on Gamma1}
 Let $g \in S_{k}(\Gamma _{1}(N))$ be a normalized newform and $L_{g}$ be the Hecke field of $g$.
We say that $F \in H_{2-k}(\Gamma_{1}(N))$ is \emph{good for $g$} if $F$ satisfies the following conditions\textrm{:}

\begin{itemize}
  \setlength{\leftskip}{0.5cm}

 \item[(1)]The principal part of $F$ at any cusp of $\Gamma_{1}(N)$ belongs to $L_{g}[q^{-1}]$.

 \item[(2)]We have that $\xi _{2-k}(F) =\dfrac{g}{\left\| g\right\| ^{2}}$.
 \end{itemize}
 \end{dfn}
 We say that a mock modular form $F^{+}$ is a \emph{good lift of $g$} if $F$ is good for $g$.

\begin{thm}[{\cite[Proposition 5.1]{bruinier2008differential}, \cite[Section 2]{candelori2017geometric}}]
Let $g \in S_{k}(\Gamma_{1}(N))$ be a normalized newform. There exists a harmonic Maass form $F \in H_{2-k}(\Gamma_{1}(N))$ such that $F$ is good for $g$.
\end{thm}

For a number field $K$, we denote the space of cusp forms whose Fourier coefficients at $\infty$ are in $K$ with character $\chi$ by $S_{k}(\Gamma_{0}(N), \chi, K)$. The spaces $M^{!}_{k}(\Gamma_{1}(N), K)$ and $S^{!}_{k}(\Gamma_{1}(N), K)$ are also defined in the same manner.

A remarkable property of a good lift $F^{+}$ is that we can modify $F^{+}$ to a $q$-series with algebraic coefficients. For a cusp form $h \in S_{k}(\Gamma_{0}(N), \chi, K)$, we define the \emph{Eichler integral of $h$} by
\begin{align*}
E_{h}(q)\coloneqq \sum_{n \geq 1}n^{1-k}a_{h}(n)q^{n}.
\end{align*}
For any complex number $\gamma \in \mathbb{C}$, we define a formal power series $\mathcal{F}_{\gamma}$ by
\begin{align*}
\mathcal{F}_{\gamma}\coloneqq F^{+}(q)-\gamma E_{g}(q).
\end{align*}

\begin{thm}[{\cite[Theorem 1.1]{guerzhoy2010p}}]\label{thm. alg of coef of mock generalver}
Let $g \in S_{k}(\Gamma_{0}(N), \chi, K)$ be a normalized newform and $F \in H_{2-k}(\Gamma_{1}(N))$ be good for $g$. Then, the $q$-series $\mathcal{F}_{a_{F^{+}}(1)}$ lies in $K(\!(q)\!)$.

\end{thm}

Let $L$ be an imaginary quadratic field, and $\chi_{L}$ be the Dirichlet character associated to $L/\mathbb{Q}$.
We say that a normalized newform $g \in S_{k}(\Gamma_{1}(N))$ has complex multiplication by an imaginary quadratic field $L$ if $a_{g}(l)=a_{g}(l)\chi_{L}(l)$ holds for prime numbers $l$ in a set of primes of density $1$. If $g$ has a complex multiplication, then all the coefficients of its good lift $F^{+}$ are algebraic without any modification.

\begin{thm}[{\cite[Corollary 1.2]{ehlen2024harmonic}}]\label{thm. coef alg}
Let $g \in S_{k}(\Gamma_{0}(N), \chi, K)$ be a normalized newform with complex multiplication by an imaginary quadratic field $L$. If $F^{+}$ is a good lift of $g$, then all the coefficients of $F^{+}$ are in $K$.
\end{thm}

We consider the Hecke action on harmonic Maass forms.
\begin{thm}\label{thm. hecke action on harmonic}
Let $g \in S_{k}(\Gamma_{1}(N))$ be a normalized newform and $F \in H_{2-k}(\Gamma_{1}(N))$ be good for $g$. There exists a weakly holomorphic modular form $h_{l} \in M_{k}^{!}(\Gamma_{1}(N))$ such that the following equality holds for all prime numbers $l$:
\begin{align*}
T_{l}F=l^{1-k}\overline{a_{g}(l)}F+h_{l}
\end{align*}
\end{thm} 
\begin{proof}
It was shown for $l\nmid N$ in \cite[Proposition 7.2]{bringmann2017harmonic}. We consider the case $l \mid N$.
In this case, we have that 
\begin{align*}
T_{l}F=U_{l}F=\dfrac{1}{l}\sum_{j=0}^{l-1}F(\dfrac{\tau+j}{l}).
\end{align*}
It is clear that $U_{l}F$ satisfies the first and second conditions of Definition \ref{def of harmonic maass}. For any $\bigl(
\begin{smallmatrix}
   a & b \\
   c & d
\end{smallmatrix}
\bigr) \in SL_{2}(\mathbb{Z})$, we show that there exist $t \in \mathbb{Z}$ and $\gamma_{j} \in SL_{2}(\mathbb{Z})$ such that the following equality holds:
\begin{align}\label{multi matrix}
\bigl(
\begin{matrix}
   1 &  j \\
   0 & l
\end{matrix}
\bigr)\bigl(
\begin{matrix}
   1 & t \\
   0 & 1
\end{matrix}
\bigr)\bigl(
\begin{matrix}
   a & b \\
   c & d
\end{matrix}
\bigr) =\gamma_{j} \bigl(
\begin{matrix}
   \alpha_{j} & \beta_{j} \\
   0 & \gamma_{j}
\end{matrix}
\bigr)
\end{align}
The left hand side of \eqref{multi matrix} is equal to $\bigl(
\begin{smallmatrix}
   a+c(t+j) & b+d(t+j) \\
   lc & ld
\end{smallmatrix}
\bigr)$. Since $(a, c)=1$ holds, there exists $t \in \mathbb{Z}$ such that $(a+c(t+j), lc)=1$. Hence, there exist $x, y \in \mathbb{Z}$ such that $\bigl(
\begin{smallmatrix}
   x & y \\
   -lc & a+c(t+j)
\end{smallmatrix}
\bigr) \in SL_{2}(\mathbb{Z})$. We define $\gamma_{j}$ by $\bigl(
\begin{smallmatrix}
   x & y \\
   -lc & a+c(t+j)
\end{smallmatrix}
\bigr)^{-1}$. Since $U_{l}F$ satisfies the first condition of Definition \ref{def of harmonic maass} and $\bigl(
\begin{smallmatrix}
   1 & t \\
   0 & 1
\end{smallmatrix}
\bigr) \in \Gamma_{0}(N)$, we have that
\begin{align*}
(U_{l}F)\mid_{2-k}\bigl(
\begin{smallmatrix}
   a & b \\
   c & d
\end{smallmatrix}
\bigr)=l^{2-\frac{3k}{2}}\sum F\mid_{2-k}\gamma_{j} \bigl(
\begin{smallmatrix}
   \alpha_{j} & \beta_{j} \\
   0 & \delta_{j}
\end{smallmatrix}
\bigr).
\end{align*}
By the definition of harmonic Maass forms, there exist $h_{j} \in \mathbb{N}$, a polynomial $P_{\gamma_{j}}$, and $\varepsilon_{j}>0$ such that
\begin{align*}
F\mid_{2-k}\gamma_{j}- P_{\gamma_{j}}(\exp(-2\pi i \tau/h_{j}))=O(e^{-\varepsilon_{j}y}).
\end{align*}
Let $h:=\lcm \{\delta_{j}h_{j} \mid j=0, \cdots l-1\}$ and $\varepsilon:=\max \{\alpha_{j}\varepsilon_{j}/\delta_{j} \mid j=0, \cdots, l-1\}$. Then there exists a polynomial $P$ such that
\begin{align*}
(U_{l}F)\mid_{2-k}\bigl(
\begin{smallmatrix}
   a & b \\
   c & d
\end{smallmatrix}
\bigr)-P(\exp(-2\pi i \tau)/h)=O(e^{-\varepsilon y}).
\end{align*}
This implies that $U_{l}F \in H_{2-k}(\Gamma_{1}(N))$.
It is clear that $\xi_{2-k}(U_{l}F)=l^{k-1}U_{l}(\xi_{2-k}(F))$ holds. Since $\xi_{2-k}$ is antilinear, we have that
\begin{align*}
U_{l}F-l^{1-k}\overline{a_{g}(l)}F \in \Ker \xi_{2-k}=M^{!}_{2-k}(\Gamma_{1}(N)).
\end{align*}
\end{proof}
\begin{corollary}
    Let $g \in S_{k}(\Gamma_{1}(N))$ be a normalized newform and $F \in H_{2-k}(\Gamma_{1}(N))$ be good for $g$. There exists a weakly holomorphic modular form $h_{l} \in M_{k}^{!}(\Gamma_{1}(N))$ such that the following equality holds for all prime numbers $l$:
    \begin{align*}
        T_{l}(D^{k-1}(F^{+}))=\overline{a_{g}(l)}D^{k-1}(F^{+})+D^{k-1}(h_{l}).
    \end{align*}
\end{corollary}

\subsection{Pairing on a quotient space of weakly holomorphic cusp forms}\label{subsection of weakly holo}
In this subsection, we define a pairing on the space $S_{k}^{!}(\Gamma_{1}(N), K)/D^{k-1}M_{2-k}^{!}(\Gamma_{1}(N), K)$ and show that an algebraic modification of $D^{k-1}(F)$ and $g$ are linearly independent in this space. 

Let $g \in S_{k}(\Gamma_{1}(N), K)$ be a normalized newform where $K$ is a number field and $F \in H_{2-k}(\Gamma_{1}(N))$ be a harmonic Maass form which is good for $g$. We recall that a formal power series $\mathcal{F}_{a_{F^{+}}(1)}$ is defined by 
\begin{align*}
\mathcal{F}_{a_{F^{+}}(1)}:=F^{+}-a_{F^{+}}(1) E_{g}=\sum _{n\gg -\infty}a_{F^{+}}(n)q^{n}-a_{F^{+}}(1)\sum_{n>0}n^{1-k}a_{g}(n)q^{n}.
\end{align*}
 By Theorem \ref{thm. alg of coef of mock generalver}, all coefficients of $\mathcal{F}_{a_{F^{+}}(1)}$ are in $K$. We note that $D^{k-1}(\mathcal{F}_{a_{F^{+}}(1)})$ is in $S_{k}^{!}(\Gamma_{0}(N), K)$ because $D^{k-1}$ gives a map from $H_{2-k}(\Gamma_{1}(N))$ to $S^{!}_{k}(\Gamma_{1}(N))$. In this subsection, we show that $g$ and $D^{k-1}(\mathcal{F}_{a_{F^{+}}(1)})$ are linearly independent in $S_{k}^{!}(\Gamma_{1}(N), K)/D^{k-1}M_{2-k}^{!}(\Gamma_{1}(N), K)$. 

\begin{dfn}
We define a pairing $\langle \cdot , \cdot \rangle$ on $S_{k}^{!}(\Gamma_{1}(N), K)/D^{k-1}M_{2-k}^{!}(\Gamma_{1}(N), K)$ by
\begin{align*}
\langle \phi, \psi \rangle:=\sum _{s \colon \text{cusp}}\sum _{n \in \mathbb{Z}\backslash\{0\}}\dfrac{a_{\phi }^{(s)}(n)a_{\psi }^{(s)}(-n)}{n^{k-1}}
\end{align*}
where $a_{\phi }^{(s)}(n)$ (resp. $a_{\psi }^{(s)}(-n)$) is the $n$-th Fourier coefficient of $\phi$ at the cusp $s$ (resp. the $(-n)$-th Fourier coefficient of $\psi$ at the cusp $s$).
\end{dfn}
The right-hand side is a finite sum because a weakly holomorphic modular form has only finitely many negative-power terms at any cusp.

\begin{remark}
P. Guerzhoy has already introduced a similar pairing when the level $N=1$ (cf. \cite{guerzhoy2008hecke}). J. Hwang and C. H. Kim also introduced a similar pairing under the technical assumption for $N$  (cf. \cite{hwang2021arithmetic}).
\end{remark}

\begin{lem}\label{lem. paring is well def}
The above pairing is well-defined.
\end{lem}
\begin{proof}
Let $ \phi \in M^{!}_{2-k}(\Gamma_{1}(N), K)$ and $\psi \in S_{k}^{!}(\Gamma_{1}(N), K)$. Then $\phi \psi$ is an element of $M_{2}^{!}(\Gamma_{1}(N), K)$. Now, we consider $\phi \psi$ as an element in $M_{2}^{!}(\Gamma_{1}(N), \mathbb{C})$. Let $\omega_{\phi \psi}$ be a meromorphic differential form on the Riemann surface $X_{1}(N)/\mathbb{C}$ corresponding to $\phi \psi$. We note that $\omega_{\phi \psi}$ has poles only in cusps. By the residue theorem, we have that $\sum_{s \colon \text{cusps}}\Res(\omega_{\phi\psi}, s)=0$. The residue of $\omega_{\phi\psi}$ at $s$ is equal to the constant term of the Fourier expansion of $\omega_{\phi\psi}$ at $s$. Therefore, we have that 
\begin{align*}
0=\sum_{s \colon \text{cusps}}\Res(\omega_{\phi\psi}, s) &=\sum _{s \colon \text{cusp}}\sum _{n \in \mathbb{Z}\backslash\{0\}}a_{\phi }^{(s)}(n)a_{\psi }^{(s)}(-n)
\\
&=\sum _{s \colon \text{cusp}}\sum _{n \in \mathbb{Z}\backslash\{0\}}\dfrac{n^{k-1}a_{\phi }^{(s)}(n)a_{\psi }^{(s)}(-n)}{n^{k-1}}=\langle D^{k-1}(\phi), \psi \rangle.
\end{align*}
\end{proof}

\begin{lem}\label{lem. const. term = peterson}
Let $F \in H_{2-k}(\Gamma_{1}(N))$ and $g \in S_{k}(\Gamma_{1}(N), \mathbb{C})$.
Then we have that 
\begin{align*}
\sum _{s \colon \text{cusp}}\sum _{n \in \mathbb{Z}\backslash\{0\}}a_{F^{+} }^{(s)}(n)a_{g}^{(s)}(-n)=(\xi_{2-k}(F), g)
\end{align*}
where the right-hand side is the Petersson inner product of $\xi_{2-k}(F)$ and $g$.
\end{lem}
\begin{proof}
See the proof of \cite[Proposition 3.5]{bruinier2004two} and \cite[Proposition 5.11]{bringmann2017harmonic}.
\end{proof}
\begin{thm}\label{linearly independent in S^!}
Let $g \in S_{k}(\Gamma_{1}(N), K)$ be a normalized newform where $K$ is a number field and $F \in H_{2-k}(\Gamma_{1}(N))$ be a harmonic Maass form which is good for $g$. Then, $g$ and $D^{k-1}(\mathcal{F}_{a_{F^{+}}(1)})$ are linearly independent in $S_{k}^{!}(\Gamma_{1}(N), K)/D^{k-1}M_{2-k}^{!}(\Gamma_{1}(N), K)$.
\end{thm}
\begin{proof}
It is enough to show that $\langle D^{k-1}(\mathcal{F}_{a_{F^{+}}(1)}), g \rangle \not=0$. By the definition of this pairing, we have that 
\begin{align*}
\langle D^{k-1}(\mathcal{F}_{a_{F^{+}}(1)}), g \rangle&=\sum _{s \colon \text{cusp}}\sum _{n <0}\dfrac{n^{k-1}a_{F^{+}}^{(s)}(n)a_{g}^{(s)}(-n)}{n^{k-1}}
\\
&=\sum _{s \colon \text{cusp}}\sum _{n <0}a_{F^{+} }^{(s)}(n)a_{g}^{(s)}(-n).
\end{align*}
 By Lemma \ref{lem. const. term = peterson}, $\langle D^{k-1}(\mathcal{F}_{a_{F^{+}}(1)}), g \rangle=(\xi_{2-k}(F), g)$ holds. We recall that $F$ is good for $g$, and therefore we have that $\xi_{2-k}(F)=g/\|g\|^{2}$ and that $\langle D^{k-1}(\mathcal{F}_{a_{F^{+}}(1)}), g \rangle=1$.
\end{proof}

\section{Cohomological interpretation of the differential of a good lift}
In this section, we review the cohomological interpretation of $D^{k-1}(F)$, where $F$ is good for $g$.
\subsection{A parabolic cohomology on the modular curve over a number field}\label{subsec. parabolic cohomology}
In this subsection, we recall the geometric setting of the modular curve and a parabolic cohomology on it. The purpose of this subsection is to define a parabolic cohomology which is isomorphic to the quotient space of $S_{k}^{!}(\Gamma_{1}(N), K)/D^{k-1}M_{2-k}^{!}(\Gamma_{1}(N), K)$. This subsection is mainly based on \cite{candelori2014harmonic}. We refer to \cite{candelori2014harmonic}, \cite{candelori2017geometric}, \cite{kazalicki2016modular}, \cite{scholl1985modular} as main references.
We consider the functor 
\begin{align*}
T : &\mathbb{Z}[1/N]\text{- Schemes}
\\
  &\longmapsto
\left\{
\begin{array}{l}
\text{ isomorphism classes of pairs } (E, \alpha : \mu_{N} \hookrightarrow E_{N}),\\
E : \text{ a generalized elliptic curve over } T, \\
\alpha : \text{ an embedding of group schemes} 
\\
\text{whose image meets every irreducible component}\\
\text{ in every geometric fiber.}
\end{array}
\right\}
\end{align*}
By \cite[Proposition 2.1]{gross1990tameness}, this functor is represented by a proper, smooth, and geometrically connected curve $X_{1}(N)$ over $\mathbb{Z}[1/N]$. For a number field $K$, we denote $X_{1}(N) \times_{\mathbb{Z}[1/N]} \Spec K$ by $X$ in this subsection.
Let $C$ be the cuspidal subscheme of $X$. $Y$ denotes $X\setminus C$ and $\mathcal{I}_{C}$ denotes the ideal sheaf of $C$. Let $\underline{\omega}$ be the relative dualizing sheaf of the universal generalized elliptic curve over $X$. Then, the sheaf $\underline{\omega}$ is an invertible sheaf and compatible with base change to any field $M$ of characteristic $0$. 
In this notation, we have canonical isomorphisms
\begin{align*}
M^{!}_{k}(\Gamma_{1}(N), K) \simeq H^{0}(Y, \underline{\omega}^{k}), \quad S_{k}(\Gamma_{1}(N), K) \simeq H^{0}(X, \underline{\omega}^{k} \otimes \mathcal{I}_{C}).
\end{align*}
Next, we consider the \emph{relative de Rham cohomology sheaf} $\mathcal{L}^{1} $ of the universal elliptic curve $ \mc{E}_{\mathrm{univ}} \rightarrow Y$. We recall that the relative de Rham cohomology sheaf $\mathcal{L}^{1} $ is defined by
\begin{align*}
\mathcal{L}^{1}\coloneqq \mathbb{R}^{1}\pi_{\ast}(0 \rightarrow \mathcal{O}_{\mc{E}_{\mathrm{univ}}} \rightarrow \Omega^{1}_{\mc{E}_{\mathrm{univ}}/Y} \rightarrow 0).
\end{align*}
It is known that $\mathcal{L}^{1}$ is a locally free sheaf of rank $2$ with a filtration :
\begin{align}\label{filtration over Y}
0 \rightarrow \underline{\omega} \rightarrow \mathcal{L}^{1} \rightarrow \underline{\omega}^{-1} \rightarrow 0
\end{align}
(see \cite[Section 3]{candelori2014harmonic}).
One of the key properties of the relative de Rham cohomology sheaf is that $\mc{L}^{1}$ has a connection. In other words, $\mathcal{L}^{1}$ has a \emph{Gauss-Manin connection}
\begin{align*}
\nabla : \mathcal{L}^{1} \rightarrow \mathcal{L}^{1} \otimes \Omega^{1}_{Y}.
\end{align*}
Moreover, it is known that we can extend $\mathcal{L}^{1}$ to a locally free sheaf of rank $2$ on $X$ and $\nabla$ to a connection with logarithmic poles
\begin{align*}
\nabla : \mathcal{L}^{1} \rightarrow \mathcal{L}^{1} \otimes \Omega^{1}_{X}(\log C)
\end{align*}
(see \cite[Section~4]{kazalicki2016modular} or \cite[Section 3]{candelori2014harmonic}).
Therefore, the connection $\nabla$ also extends to the symmetric power $\mathcal{L}^{k}$ of $\mathcal{L}^{1}$
\begin{align*}
\nabla : \mathcal{L}^{k} \rightarrow \mathcal{L}^{k} \otimes \Omega^{1}_{X}(\log C)
\end{align*}
by the following formula:
\begin{align*}
\nabla(\eta_{1}\otimes \eta_{2})=\nabla(\eta_{1})\otimes \eta_{2}+\eta_{1} \otimes \nabla(\eta_{2}).
\end{align*}
By using the connection $\nabla$, we define a parabolic cohomology on $X$.
\begin{dfn}[{\cite{castella2018exceptional}}]\label{def. parabolic}
We define the parabolic cohomology $\mathbb{H}^{1}_{\mathrm{par}}(X, \mathcal{L}^{k})$ by
\begin{align*}
\Im \left(\mathbb{H}^{1}(\mathcal{L}^{k} \otimes \mathcal{I}_{C} \rightarrow \mathcal{L}^{k} \otimes \Omega^{1}_{X}) \rightarrow \mathbb{H}^{1}(\mathcal{L}^{k} \rightarrow \mathcal{L}^{k} \otimes \Omega^{1}_{X}(\log C)) \right).
\end{align*}
\end{dfn}

\begin{remark}
L. Candelori defined the parabolic cohomology $\mathbb{H}^{1}_{\mathrm{par}}(X, \mathcal{L}^{k})$ by another complex in his paper (see \cite[Section 5]{candelori2014harmonic}).
\end{remark}

L. Candelori \cite{candelori2014harmonic} showed the following theorem.

\begin{thm}[{\cite[Theorem 6]{candelori2014harmonic}}]
For all integers $k \geq 2$, there is the following canonical isomorphism:
\begin{align*}
\mathbb{H}^{1}_{\mathrm{par}}(X, \mathcal{L}^{k-2}) \simeq \dfrac{S^{!}_{k}(\Gamma_{1}(N), K)}{D^{k-1}M^{!}_{2-k}(\Gamma_{1}(N), K)}.
  \end{align*}
\end{thm}
For a weakly holomorphic cusp form $\phi$, we denote the class in $\mathbb{H}^{1}_{\mathrm{par}}(X, \mathcal{L}^{k-2})$ corresponding to $\phi$ by $[\phi]$.
We denote the $g$-isotypical component of $\mathbb{H}^{1}_{\mathrm{par}}(X, \mathcal{L}^{k-2})$ by $M_{\mathrm{dR}}(g)$. It is known that $M_{\mathrm{dR}}(g)$ is a $2$-dimensional $K$ vector space.
Next, we give a basis of $M_{\mathrm{dR}}(g)$ when all Fourier coefficients of $g$ are real.

\begin{thm}\label{thm. F^+ g are basis}
    If all the Fourier coefficients of $g$ are real, then $\{[D^{k-1}(\mc{F}_{a_{F^{+}}(1)})], g\}$ is a basis of $M_{\mathrm{dR}}(g)$.
\end{thm}
\begin{proof}
Firstly, we note that $D^{k-1}(\mc{F}_{a_{F^{+}}(1)})$ is in $S_{k}^{!}(\Gamma_{1}(N), K)/D^{k-1}M_{2-k}^{!}(\Gamma_{1}(N), K)$. Therefore, $[D^{k-1}(\mc{F}_{a_{F^{+}}(1)})]$ belongs to $\mathbb{H}^{1}_{\mathrm{par}}(X, \mathcal{L}^{k-2})$. By Theorem \ref{thm. hecke action on harmonic} and $a_{g}(l) \in \R$, Hecke eigenvalues of $D^{k-1}(\mc{F}_{a_{F^{+}}(1)})$ coincide with that of $g$. This implies that $[D^{k-1}(\mc{F}_{a_{F^{+}}(1)})] \in M_{\mathrm{dR}}(g)$. Finally, we show that $[D^{k-1}(\mc{F}_{a_{F^{+}}(1)})]$ and $g$ are linearly independent. It is enough to show that $D^{k-1}(\mc{F}_{a_{F^{+}}(1)})$ and $g$ are linearly independent in $S_{k}^{!}(\Gamma_{1}(N), K)/D^{k-1}M_{2-k}^{!}(\Gamma_{1}(N), K)$. This is already proved in Theorem \ref{linearly independent in S^!}.
\end{proof}
Therefore, the algebraic modification of a good lift can be regarded as an element of the cohomology group over the modular curve.

\subsection{A parabolic cohomology on the modular curve over a $p$-adic field}\label{subsec. par p-adic}
In this subsection, we treat the basic results on a parabolic cohomology on the modular curve over $K_{p}$ and define an endomorphism $V$ on $M_{\mathrm{dR}, p}(g)\coloneqq M_{\mathrm{dR}}(g) \otimes K_{p}$ following \cite{candelori2017geometric}. This subsection is mainly based on \cite[Section 3]{candelori2017geometric}.

We denote the ring of integers of $K_{p}$ by $R_{p}$. Let $X_{R_{p}}\coloneqq X_{1}(N) \times_{\mathbb{Z}[1/N]} R_{p}$ and $X_{K_{p}}\coloneqq  X_{R_{p}} \times_{R_{p}} K_{p}$. For simplicity, we denote $X_{K_{p}}$ by $X$ throughout this subsection. Let $E_{p-1} \in H^{0}(X_{R_{p}}, \underline{\omega}^{p-1})$ be the global section given by the normalized Eisenstein series of weight $p-1$ and level $1$. For any $\varepsilon \in |R_{p}|$, we define a rigid analytic space $X_{(\varepsilon)}$ by 
\begin{align*}
(X_{(\varepsilon)})^{\text{cl}}=\{x \in X^{\text{cl}} \mid \left| E_{p-1}(x) \right| > \varepsilon \}.
\end{align*}
We also define the \emph{ordinary locus} $X^{\text{ord}}$ of $X$ by
\begin{align*}
(X^{\text{ord}})^{\text{cl}}=\{x \in X^{\text{cl}} \mid \left| E_{p-1}(x) \right| \geq 1 \}.
\end{align*}
Since $\left| E_{p-1}(c) \right| \geq 1$ for every cusp $c$, the cuspidal subscheme $C$ is in $X_{(\varepsilon)}$ for all $\varepsilon \in |R_{p}|$. Therefore, we define the rigid analytic spaces $Y^{\text{ord}}$ and $Y_{(\varepsilon)}$ by
\begin{align*}
Y^{\text{ord}}\coloneqq X^{\text{ord}} \setminus C, \quad Y_{(\varepsilon)}\coloneqq X_{(\varepsilon)} \setminus C.
\end{align*}
By using this notation, we define an overconvergent modular form following \cite{candelori2017geometric}.
\begin{dfn}[{\cite[definition 3.1]{candelori2017geometric}}]
An \emph{overconvergent modular form of weight $k$} is a rigid analytic section of $\underline{\omega}^{k}$ on $Y_{(\varepsilon)}$ for some $\varepsilon <1$.
\end{dfn}
For any wide-open neighborhood $W$ of $X^{\text{ord}}$, we denote $W\setminus C$ by $W^{\circ}$. We denote a cohomology group $\mathbb{H}^{1}(W^{\circ}, \mathcal{L}^{k} \xrightarrow{\nabla} \mathcal{L}^{k} \otimes \Omega^{1}_{X} )$ by $\mathbb{H}^{1}(W^{\circ}, \mathcal{L}^{k})$. Since $H^{n}(W^{\circ}, \mathcal{G})$ vanishes for all $n>0$ and all coherent sheaves $\mathcal{G}$ on $W^{\circ}$, we have an isomorphism
\begin{align*}
\mathbb{H}^{1}(W^{\circ}, \mathcal{L}^{k}) \simeq H^{0}(W^{\circ}, \mathcal{L}^{k} \otimes \Omega^{1}_{X})/\nabla H^{0}(W^{\circ}, \mathcal{L}^{k}).
\end{align*}
Next theorem is important to construct an overconvergent modular form from a good lift of $g$.
\begin{thm}[{\cite[Theorem 3.3]{candelori2017geometric}}]
For every $k \geq 0$, there is a linear map
\begin{align*}
\theta^{k+1} : H^{0}(W^{\circ}, \underline{\omega}^{-k}) \rightarrow H^{0}(W^{\circ}, \underline{\omega}^{k+2})
\end{align*}
whose action on $q$-expansions is $D^{k+1}=(qd/dq)^{k+1}$.
In addition, the natural injection
\begin{align*}
H^{0}(W^{\circ}, \underline{\omega}^{k+2})\simeq H^{0}(W^{\circ}, \underline{\omega}^{k} \otimes \Omega^{1}_{X}) \hookrightarrow H^{0}(W^{\circ}, \mathcal{L}^{k} \otimes \Omega^{1}_{X})   
\end{align*}
induces an isomorphism
\begin{align*}
\mathbb{H}^{1}(W^{\circ}, \mathcal{L}^{k}) \simeq H^{0}(W^{\circ}, \underline{\omega}^{k+2}) /\theta^{k+1}H^{0}(W^{\circ}, \underline{\omega}^{-k}).
\end{align*}
\end{thm}
Hence, if an element of $\mathbb{H}^{1}(W^{\circ}, \mathcal{L}^{k})$ vanishes, then it is in the image of $\theta^{k+1}$.

Now, we consider two distinguished wide-open neighborhoods $W_{1}\coloneqq X_{(p^{-p/p+1})}$ and $W_{2}\coloneqq X_{(p^{-1/p+1})}$ of $X^{\text{ord}}$. We note that $W_{2} \subset W_{1}$ holds. There are two operators 
\begin{align*}
U : H^{0}(W_{2}, \underline{\omega}^{k}) \rightarrow H^{0}(W_{1}, \underline{\omega}^{k}), \quad V : H^{0}(W_{1}, \underline{\omega}^{k}) \rightarrow H^{0}(W_{2}, \underline{\omega}^{k})
\end{align*}
whose actions on $q$-expansions are given by
\begin{align*}
U\left(\sum a(n)q^{n}\right)=\sum a(pn)q^{n}, \quad V\left(\sum a(n)q^{n}\right)=\sum a(n)q^{pn}.
\end{align*}
By a restriction map, we can regard a normalized newform $g \in S_{k}(\Gamma_{1}(N), K)$ as a section of $\underline{\omega}^{k}$ over $W_{1}$. Therefore, $V(g)$ can be regarded as an element of $\mathbb{H}^{1}(W_{2}^{\circ}, \mathcal{L}^{k-2})$. Since $\mathbb{H}^{1}_{\mathrm{par}}(X_{1}(N)/K, \mathcal{L}^{k-2})$ is stable under base change to a field with characteristic $0$, the space $M_{\mathrm{dR}, p}(g)\coloneqq M_{\mathrm{dR}}(g) \otimes K_{p}$ is a subspace of $\mathbb{H}^{1}_{\mathrm{par}}(X, \mathcal{L}^{k-2})$.
By \cite[Theorem 4.2]{coleman1989reciprocity} and \cite[Proposition 10.3]{coleman1994p}, the following restriction map is injective:
\begin{align*}
\mathbb{H}^{1}_{\mathrm{par}}(X, \mathcal{L}^{k-2}) \rightarrow \mathbb{H}^{1}(W_{2}^{\circ}, \mathcal{L}^{k-2}).
\end{align*}

Therefore, $M_{\mathrm{dR}, p}(g)$ is a subspace of $\mathbb{H}^{1}(W_{2}^{\circ}, \mathcal{L}^{k-2})\simeq H^{0}(W_{2}^{\circ}, \underline{\omega}^{k}) /\theta^{k-1}H^{0}(W_{2}^{\circ}, \underline{\omega}^{-k+2})$ and the restriction of $V$ defines an endomorphism of $M_{\mathrm{dR}, p}(g)$.

\section{$p$-adic properties of mock modular forms}

In this section, we recall $p$-adic properties of mock modular forms. In particular, we define the $p$-adic constants $\gamma_{g}$ and $\delta_{g}$. Throughout this section, we assume that $g \in S_{k}(\Gamma_{0}(N), \chi, K)$ is a normalized newform with real Fourier coefficients. Let $\beta, \beta'$ be the roots of $X^{2}-a_{g}(p)X+\chi(p)p^{k-1}$ satisfying $v_{p}( \beta ) \leq v_{p}( \beta')$.  Let $F$ be a harmonic Maass form that is good for $g$.

\subsection{$p$-adic modular form associated to a mock modular form}
In this subsection, we recall how to construct a $p$-adic modular form from a good lift $F^{+}$ following \cite{bringmann2012mock}. We refer to \cite{bringmann2012mock}, \cite{candelori2017geometric}, \cite{guerzhoy2014zagier}, \cite{guerzhoy2010p} as main references.

Firstly, we recall a definition of a $p$-adic modular form. For two formal power series 
\begin{align*}
A(q)=\sum_{n\geq -t} a_{A}(n)q^{n}\quad \text{and} \quad B(q)=\sum_{n\geq -t} a_{B}(n)q^{n} \in \mathbb{C}_{p}(\!(q)\!),
\end{align*}
we denote $A(q) \equiv B(q) \pmod{p^{m}}$ when 
\begin{align*}
\min_{n \geq -t} v_{p}(a_{A}(n)-a_{B}(n)) \geq m.
\end{align*}
\begin{dfn}
Let $H(q)=\sum_{n\geq -t} a_{H}(n)q^{n} \in \mathbb{C}_{p}(\!(q)\!)$. We say that $H(q)$ is a \emph{$p$-adic modular form of weight $k$ and level $M$ with character $\chi$} if $H(q)$ satisfies the following condition.
For every $m \in \mathbb{N}$, there exists 
\begin{align*}
H_{m}(q)=\sum_{n\geq -t} a_{H_{m}}(n)q^{n} \in M_{k}^{!}( \Gamma _{0}(M), \chi, \bar{\mathbb{Q}})
\end{align*}
such that $H(q) \equiv H_{m}(q) \; \pmod{p^{m}}$ holds.
\end{dfn}
We note that the smallest exponent $t$ of $H_{m}(q)$ is independent of $m$. In subsection \ref{subsec. par p-adic}, we define two operators $U$ and $V$. Following their actions on $q$-expansions, we define two operators $U_{p}$ and $V_{p}$ on $\mathbb{C}_{p}(\!(q)\!)$ by
\begin{align*}
& \left(\sum_{n \in \mathbb{Z}} a(n)q^{n}\right)\vert U_{p}\coloneqq \sum_{n \in \mathbb{Z}} a(pn)q^{n},
\\
& \left(\sum_{n \in \mathbb{Z}} a(n)q^{n}\right)\vert V_{p}\coloneqq \sum_{n \in \mathbb{Z}} a(n/p)q^{n}=\sum_{n \in \Z}a(n)q^{pn}.
\end{align*}
We note that $U$ and $V$ are operators on the cohomology group but $U_{p}$ and $V_{p}$ are operators on $q$-series.
We recall that the Eichler integral $E_{h}$ of a $q$-series $h(q)=\sum_{n>0}a_{h}(n)q^{n}$ is defined by
\begin{align*}
    E_{h}(q)\coloneqq \sum_{n>0}n^{1-k}a_{h}(n)q^{n}.
\end{align*}
For a complex number $a_{F^{+}}(1) \notin K$, we define a set $a_{F^{+}}(1)+K_{p}$ as follows:
\begin{align*}
a_{F^{+}}(1)+K_{p}\coloneqq \{(a_{F^{+}}(1), \alpha) \mid \alpha \in K_{p}\}.
\end{align*}
We denote an element $(a_{F^{+}}(1), \alpha) \in a_{F^{+}}(1)+K_{p}$ by $a_{F^{+}}(1)+\alpha$. When $a_{F^{+}}(1) \in K$, we define $a_{F^{+}}(1)+K_{p}$ by $K_{p}$. For any $\gamma=a_{F^{+}}(1)+\alpha \in a_{F^{+}}(1)+K_{p}$, we define a $q$-series $F^{+}(q)-\gamma E_{g}(q)$ by
\begin{align*}
F^{+}(q)-\gamma E_{g}(q)\coloneqq (F^{+}(q)-a_{F^{+}}(1)E_{g}(q))-\alpha E_{g}(q).
\end{align*}
By Theorem \ref{thm. alg of coef of mock generalver}, the $q$-series $F^{+}(q)-\gamma E_{g}(q)$ lies in $K_{p}(\!(q)\!)$ for all $\gamma \in a_{F^{+}}(1)+K_{p}$. 
For $\gamma \in a_{F^{+}}(1)+K_{p}$ and $\delta \in K_{p}$, we define a formal power series $\mathcal{F}_{\gamma, \delta}$ by
\begin{align*}
\mathcal{F}_{\gamma, \delta}\coloneqq F^{+}-\gamma E_{g}-\delta E_{g \vert V_{p}}.
\end{align*}
K. Bringmann et~al.\@ \cite{bringmann2012mock} showed that $\mathcal{F}_{\gamma, \delta}$ is a $p$-adic modular form for exactly one pair $(\gamma, \delta)$.
\begin{thm}[{\cite[Theorem 1.1 (2)]{bringmann2012mock}}]\label{thm. p-adic mock gene}
We assume that $v_{p}(\beta)<v_{p}(\beta')$ and $v_{p}(\beta')\not=k-1$.
Then, there exist unique $p$-adic numbers $\gamma_{g} \in a_{F^{+}}(1)+K_{p}$ and $\delta_{g} \in K_{p}$ such that $\mathcal{F}_{\gamma_{g}, \delta_{g}}$ is a $p$-adic modular form. 
\end{thm}
Next, we consider the case where $g$ has a complex multiplication.
From now on, and until the end of this subsection, we assume that $g$ has complex multiplication by an imaginary quadratic field $L$ and that $p$ is inert in $L$. By Theorem \ref{thm. coef alg}, we have $a_{F^{+}}(1) \in K$ and $a_{F^{+}}(1)+K_{p}=K_{p}$.

Firstly, we compute the value of the limit of $a_{D^{k-1}(F)}( p^{2m}) /\beta ^{2m}$ as $m \to \infty$. 

\begin{proposition}[\cite{bringmann2012mock}]\label{prop. lim=0}
We have that 
\begin{align*}
\lim _{m\rightarrow \infty }\dfrac{a_{D^{k-1}(F)}( p^{2m}) }{\beta ^{2m}}=0.
\end{align*} 
\end{proposition}

\begin{proof}
See the proof of \cite[Proposition 1.4]{bringmann2012mock}. 
\end{proof}
On the other hand, the limit $\lim _{m\rightarrow \infty }a_{D^{k-1}(F)}(p^{2m+1})/\beta ^{2m}$ is important.
K. Bringmann et~al.\@ \cite{bringmann2012mock} showed the following theorem:

\begin{thm}[{\cite[Theorem 1.3]{bringmann2012mock}}]\label{def alpha_g}
There exists a unique $p$-adic constant $\delta_{g} \in K_{p}$ such that $\mc{F}_{0, \delta_{g}}$ is a $p$-adic modular form of weight $2-k$ and level $pN$. Furthermore, $\delta_{g}$ is given by the following $p$-adic limit: 
\begin{align*}
\delta_{g}=\displaystyle\lim _{m\rightarrow \infty }\dfrac{a_{D^{k-1}(F)}(p^{2m+1}) }{\beta ^{2m}}.
\end{align*}
\end{thm}

\subsection{Overconvergent modular form associated to a mock modular form}
In this subsection, we recall the construction of an overconvergent modular form from a good lift of $g$ in \cite{candelori2017geometric}. This subsection contains no new results. We refer to \cite{candelori2017geometric} as main reference.

In Subsection \ref{subsec. par p-adic}, we treat $M_{\mathrm{dR}, p}(g)=M_{\mathrm{dR}}(g) \otimes K_{p}$ and saw that $M_{\mathrm{dR}, p}(g)$ naturally lies in $\mathbb{H}^{1}(W_{2}^{\circ}, \mathcal{L}^{k-2})$. Hence, there is a linear operator $V$ on $M_{\mathrm{dR}, p}(g)$ whose action on the $q$-expansion is $V_{p}$. When $\{g, V(g)\}$ is a basis of $M_{\mathrm{dR}, p}(g)$, we can construct an overconvergent modular form from $D^{k-1}(\mc{F}_{a_{F^{+}}(1)})$.
Firstly, we consider the condition that $\{g, V(g)\}$ is a basis of $M_{\mathrm{dR}, p}(g)$.

\begin{proposition}[{\cite[Proposition 3.4]{candelori2017geometric}}]
    If the following conditions hold, then $\{g, V(g)\}$ is a basis of $M_{\mathrm{dR}, p}(g)$:
    \begin{itemize}
  \setlength{\leftskip}{0.5cm}

 \item[(1)]$\beta \not=\beta'$

 \item[(2)]$g_{\beta'} \notin \Im(\theta^{k-1})$
 \end{itemize}
\end{proposition}

\begin{remark}\label{remark cm basis}
By \cite[Lemma 6.3]{coleman1996classical}, $g_{\beta'} \notin \Im(\theta^{k-1})$ holds if $v_{p}(\beta')<k-1$. When $g$ has complex multiplication by an imaginary quadratic field in which $p$ is inert, $\beta=-\beta'$ and $v_{p}(\beta)=v_{p}(\beta')=(k-1)/2$ hold. Hence, if $g$ satisfies the assumptions of Theorem \ref{thm. p-adic mock gene} or Theorem \ref{def alpha_g}, then $\{g, V(g)\}$ is a basis of $M_{\mathrm{dR}, p}(g)$.
\end{remark}

Since $[D^{k-1}(\mc{F}_{a_{F^{+}}(1)})] \in M_{\mathrm{dR}}(g)$ holds, there exist unique $p$-adic numbers $A, B \in K_{p}$ such that the following equality holds in $M_{\mathrm{dR}, p}(g)$:
\begin{align*}
    [D^{k-1}(\mc{F}_{a_{F^{+}}(1)})]=Ag+BV(g).
\end{align*}
We recall that $M_{\mathrm{dR}, p}(g)$ lies in $\mathbb{H}^{1}(W_{2}^{\circ}, \mathcal{L}^{k-2})\simeq H^{0}(W_{2}^{\circ}, \underline{\omega}^{k}) /\theta^{k-1}H^{0}(W_{2}^{\circ}, \underline{\omega}^{-k+2})$. Therefore, there exists an element $G$ of $H^{0}(W_{2}^{\circ}, \underline{\omega}^{-k+2})$ such that
\begin{align*}
    D^{k-1}(\mc{F}_{a_{F^{+}}(1)})-Ag-BV(g)=\theta^{k-1}(G).
\end{align*}
Since the action of $\theta$ on the $q$-expansiton is $qd/dq$, the following is a $q$-expansion of an overconvergent modular form:
\begin{align*}
    F^{+}-a_{F^{+}}(1)E_{g}-AE_{g}-BE_{g \vert V_{p}}.
\end{align*}
Hence, we have the following theorem.

\begin{thm}[{\cite{candelori2017geometric}}]\label{thm. over mock}
    We assume that $\{g, V(g)\}$ is a basis of $M_{\mathrm{dR}, p}(g)$. Then there exist unique $\gamma_{g} \in a_{F^{+}}(1)+K_{p}$ and $\delta_{g} \in K_{p}$ such that the following $q$-series is a $q$-expansion of an overconvergent modular form:
    \begin{align*}
\mathcal{F}_{\gamma_{g}, \delta_{g}}= F^{+}-\gamma_{g} E_{g}-\delta_{g} E_{g \vert V_{p}}.
    \end{align*}
\end{thm}

We note that $p$-adic constants $\gamma_{g}$ and $\delta_{g}$ in Theorem \ref{thm. over mock} are the same as those in Theorem \ref{thm. p-adic mock gene} and Theorem \ref{def alpha_g} because an overconvergent modular form is a $p$-adic modular form.

\section{Proof of the main theorems}
Let $K$ be a number field and $g \in S_{k}(\Gamma_{0}(N), \chi, K)$ be a normalized newform. Throughout this section, we assume that $p\nmid N$ and all the Fourier coefficients of $g$ are real. Let $F \in H_{2-k}(\Gamma_{1}(N), K)$ be a harmonic Maass form that is good for $g$.

\begin{thm}
 Let $g \in S_{k}(\Gamma_{0}(N), \chi, K)$ be a normalized newform with real Fourier coefficients. Suppose that $p\nmid 6N$, and $N \geq 5$. We also assume that $\{g, V(g)\}$ is a basis of $M_{\mathrm{dR}, p}(g)$. Then, $\delta_{g} \not=0$ holds.
\end{thm}
\begin{proof}
Since all Fourier coefficients of $g$ are real, $D^{k-1}(\mc{F}_{a_{F^{+}}(1)})$ is in $M_{\mathrm{dR}}(g)$. By the assumption that $\{g, V(g)\}$ is a basis of $M_{\mathrm{dR}, p}(g)$, there exist unique $\gamma_{g} \in a_{F^{+}}(1)+K_{p}$ and $\delta_{g} \in K_{p}$ such that $\mc{F}_{\gamma_{g}, \delta_{g}}$ is an overconvergent modular form. In particular, we write $\gamma_{g}=a_{F^{+}}(1)+\alpha_{g}$ for some $p$-adic number $\alpha_{g}$. By the definition of $\gamma_{g}$ and $\delta_{g}$, the following equality holds in $M_{\mathrm{dR}, p}(g)$:
\begin{align*}
    [D^{k-1}(\mc{F}_{a_{F^{+}}(1)})]=\alpha_{g}g+\delta_{g}V(g).
\end{align*}
Hence, it is enough to show that $[D^{k-1}(\mc{F}_{a_{F^{+}}(1)})]$ and $g$ are linearly independent in $M_{\mathrm{dR}}(g)$. This is already proved in Theorem \ref{thm. F^+ g are basis}.
\end{proof}

Next, we consider the case where $g$ has complex multiplication. From now on, we assume that $g \in S_{k}(\Gamma_{0}(N), \chi,  K)$ has complex multiplication by an imaginary quadratic field $L$.  Then, $D^{k-1}(F) \in S^{!}_{k}(\Gamma_{1}(N), K)$ by Theorem \ref{thm. coef alg}.

\begin{corollary}
 Let $g \in S_{k}(\Gamma_{0}(N), \chi, K)$ be a normalized newform with real Fourier coefficients. Suppose that $p\nmid 6N$, and $N \geq 5$. We also assume that $g$ has complex multiplication by an imaginary quadratic field in which $p$ is inert. Then, $\delta_{g} \not=0$ holds.
\end{corollary}
\begin{proof}
This is clear because $\{g, V(g)\}$ is a basis of $M_{\mathrm{dR}, p}(g)$ when $g$ has complex multiplication by an imaginary quadratic field in which $p$ is inert by Remark \ref{remark cm basis}.
\end{proof}

\bibliographystyle{amsplain}
\bibliography{References}

\end{document}